\documentclass[12pt]{amsart} 
\usepackage{amssymb} 
\usepackage{geometry}
\usepackage[utf8]{inputenc}
\usepackage[english]{babel}
\usepackage{lineno,hyperref}
\usepackage{amsmath}
\usepackage{amscd}
\usepackage{amsthm}
\usepackage{amsfonts}
\usepackage{color}
\usepackage{xfrac}
\usepackage{mathtools}
\usepackage{graphicx}
\usepackage{diagrams} 
\usepackage{tikz}
\usetikzlibrary{arrows}
\usepackage{bbold}
\usepackage[T1]{fontenc}
\newtheorem{theorem}{Theorem}[section]

\newtheorem{corollary}[theorem]{Corollary}
\newtheorem{lemma}[theorem]{Lemma}

\newtheorem{remark}[theorem]{Remark}

\newtheorem{summary}[theorem]{Summary}

\numberwithin{equation}{section}

\AtBeginDocument{\setlength{\linenumbersep}{\dimexpr\oddsidemargin+0.9in-\linenumberwidth\relax}}

\begin{document}

\title[Geometric representation of $\mathbb{L}$-homology classes]{On geometric representation of $\mathbb{L}$-homology classes} 
\author[F. Hegenbarth and D.D. Repov\v{s}]{Friedrich Hegenbarth and Du\v{s}an D. Repov\v{s}}
\address{Dipartimento di Matematica "Federigo Enriques", Universit\` a degli studi di Milano, 20133 Milano, Italy}\email{friedrich.hegenbarth@unimi.it} 
\address{Faculty of Education and Faculty of Mathematics and Physics, University of Ljubljana \& Institute of Mathematics, Physics and Mechanics, 1000 Ljubljana, Slovenia}\email{dusan.repovs@guest.arnes.si}
\begin{abstract}  
In this 
chapter
 we
 give a geometric representation of  $H_{n}(B;\mathbb{L})$ classes, where $\mathbb{L}$ is the $4$-periodic surgery spectrum, by establishing a relationship between 
 the normal cobordism classes
${\mathcal{N}}^{H}_{n}(B,\partial)$ 
and the $n$-th $\mathbb{L}$-homology of $B$,
 representing
the elements of 
$H_{n}(B;\mathbb{L})$
by normal degree one maps with a reference map to $B$.   
More precisely, we prove that for every $n \ge 6$ and every finite complex $B,$ there exists a map
$\Gamma: H_n(B;\mathbb{L}) \longrightarrow \mathcal{N}^{H}_{n}(B,\partial).$
\end{abstract}
\date{}  
\maketitle

\section{Introduction}\label{s1}   

Giving a meaning
  to  algebraic objects  (e.g., homology classes  of generalized
homology theories) in terms of geometric objects, is a fundamental task in 
algebraic topology.
 The 
 relation
  between singular homology 
and classical  cobordism theory is  a well-known example.

 In this paper we shall consider 
 normal degree one maps\index{normal degree one map}
  $X^n \to M^n$
with a
 reference map\index{reference map}
$q: M^n \to B,$
where
$X^n$
is either a generalized or a topological $n$-manifold,
$M^n$ is a topological $n$-manifold, 
and
$B$
is a finite complex.
 We shall denote 
 normal cobordism classes\index{normal cobordism class}
  of such objects
$X^n \to M^n$
with
${\mathcal{N}}^{H}_{n}(B)$ 
(resp., 
${\mathcal{N}}_{n}(B)$). 
We emphasize that 
${\mathcal{N}}^{H}_{n}(B)$ 
(resp., 
${\mathcal{N}}_{n}(B)$)
should not be confused with the structure set in the case when 
$B$
is a
topological $n$-manifold (or a $PD_n$-complex).
It is easy to see that in our definition,
\[
\mathcal{N}_{n}(B) \cong \Omega_{n}(B \times G/TOP),
\]
where
$\Omega_{n}(\cdot)$
denotes  $n$-dimensional cobordisms. 

A  {\it 
generalized $n$-manifold\index{generalized manifold}
$X^n$ (with boundary)} is an $n$-dimensional
 Euclidean neighborhood retract\index{Euclidean neighborhood retract}
  (ENR) with the local homology
\[
H_* (X^n, X^n \setminus \{x\};\mathbb{Z}) \cong H_* (\mathbb{R}^{n}, \mathbb{R}^{n} \setminus \{0\};\mathbb{Z}),   \hbox{\rm{for every}} \ x \in X^n.
\]
The {\it boundary} $\partial X^n$ of $X^n$ is  defined by 
\[
\partial X^n = \{ x \in X^n: H_n (X^n, X^n \setminus \{x\};\mathbb{Z}) \cong 0 \}.
 \]
Then $\partial X^n$ is a generalized $(n-1)$-manifold without boundary (see {\sc Mitchell}~\cite{M}). Also,
 $X^n$ (resp. $\partial X^n$) is an $n$-dimensional (resp. $(n-1)$-dimensional) 
 Poincar\'{e} duality space\index{Poincar\'{e} duality space}. 
We shall consider generalized manifolds with boundary as pairs
of $(n+1)$-dimensional  
(resp. $n$-dimensional)
generalized manifolds $(W^{n+1},\partial W^{n+1}).$

For every
$n \ge 6,$
{\sc Quinn}~\cite{Qu83} has constructed an obstruction
$i(X^n)\in \mathbb{Z},$
which vanishes if there exists a 
cell-like map\index{cell-like map}
$M^n \to X^n$
of $M^n$ onto $X^n,$
where
$M^n$
is a closed topological $n$-manifold (for more on cell-like maps see the survey {\sc Mitchell-Repov\v{s}}~\cite{MR}).
The construction is done locally, in particular
$i(U)=i(X^n),$
for every open subset
$U\subset X^n.$
Because of its nice properties, one considers
$I(X^n)=1+8i(X^n),$
which we shall hereafter call the 
{\it 
Quinn index\index{Quinn index}
} of $X^n$.
In particular, 
$I(W^{n+1})=I(\partial W^{n+1}),$
which follows from the product property of $I(X^n)$. For more on this 
topic
 see the monograph
{\sc Cavicchioli-Hegenbarth-Repov\v{s}}~\cite{CHR}.

Our  main result (to be proved in Section~\ref{s4}) establishes a relationship between 
${\mathcal{N}}^{H}_{n}(B)$ 
and the $n$-th 
$\mathbb{L}$-homology\index{$\mathbb{L}$-homology}
of a 
finite complex\index{finite complex}
$B$, i.e.
$H_{n}(B;\mathbb{L}),$
where
$\mathbb{L}$
denotes the 
$4$-periodic surgery spectrum\index{surgery spectrum},
 i.e.
\[
\mathbb{L}_{0} \cong \mathbb{Z}\times G/TOP 
\]
(see {\sc Ranicki}~\cite{Ra92}).
Namely, we represent
the elements of 
$H_{n}(B;\mathbb{L})$
by normal degree one maps with a 
reference map
 to $B$, i.e.
as
the objects of the following type
\begin{equation}\label{d1}
\begin{tikzpicture}[baseline=-1.2cm, node distance=3.0cm, auto, ]
  \node (Xl) {$X^{n}$};
  \node (Fl) [right of=Xl] {$N^n$}; 
  \node (Fl1) [below of=Fl, yshift=1cm] {$B$};
  \draw[->, font=\small] (Xl) to node [midway, above]{} (Fl); 
  \draw[->, font=\small] (Fl) to node {$q$} (Fl1); 
\end{tikzpicture}
\end{equation} 
where $X^n$ is a generalized $n$-manifold with boundary and $N^n$ is a topological $n$-manifold with boundary, and
normal cobordism classes of such objects are denoted by
$\mathcal{N}^{H}_{n}(B,\partial)$
(resp.
$\mathcal{N}_{n}(B,\partial))$. 

\begin{theorem}\label{th2}
For every $n \ge 6$ and every finite complex $B,$ there exists a map
\[
\Gamma: H_n(B;\mathbb{L}) \longrightarrow \mathcal{N}^{H}_{n}(B,\partial).
\]
\end{theorem}

\section{Geometric representation of the elements of $H_n(B;\mathbb{L})$}\label{s3}
Let $B$ be a finite complex, 
$n \ge 6,$
 and let 
$H_n(B;\mathbb{L})$
be the 
Steenrod homology\index{Steenrod homology}
 of $B$ with respect to the 
$4$-periodic surgery spectrum
$\mathbb{L}.$
The Steenrod homology behaves well on the category of all compact metric spaces - see {\sc Kahn-Kaminker-Schochet}~\cite{KKS77}.
The spectrum
$\mathbb{L}$
is algebraically defined. It is an
$\Omega$-spectrum\index{$\Omega$-spectrum},
 i.e.
it is a sequence of
$\Delta$-sets\index{$\Delta$-set} 
$\mathbb{L}_{q},$
where
$\mathbb{L}_{q}$
is homotopy equivalent to
$\Omega\mathbb{L}_{q-1}$
with
$\mathbb{L}_{0}\cong \mathbb{Z} \times G/TOP$
as
$\Delta$-sets. 
Each
$\mathbb{L}_{q}$
consists of a sequence
$\mathbb{L}_{q}<j>.$

We shall follow {\sc Ranicki}~\cite[Chapter 12]{Ra92} to define elements of 
$H_n(B;\mathbb{L}).$
To this end, we have to embed
$B \subset S^m,$ where $m$ is sufficiently large, and consider the 
dual complexes\index{dual complex}
 of $B$ and $S^m,$ 
which will be denoted
respectively as 
$\bar B$
and $\Sigma^m.$ 
However, we shall keep the notation
$B$ and $S^m$ also for the duals.

Let 
$(V^m,W^m)$ 
be a pair of simplicial complexes homotopy equivalent to the pair
$(S^m,S^m \setminus B).$
An element
$x \in H_n(B;\mathbb{L})$
is then given by a simplicial map
\[  
u:(V^m,W^m) \longrightarrow (\mathbb{L}_{n-m},*),
\]
which sends $W^m$ to $0$ and satisfies a certain cycle condition.
Clearly, $u$ is well-defined by 
$x,$
up to some equivalence (i.e. boundary)
condition.
Moreover, if
$\sigma \in B,$
then
\[
u(\sigma) \in \mathbb{L}_{n-m}< m- | \sigma | >,
\]
where
$| \sigma |$
denotes
 the dimension of $\sigma,$
 i.e. 
 $m-| \sigma |$
 is the dimension of its dual cell 
 $D(\sigma,V^m)$
 in
 $V^m.$
 
 As described in
 {\sc Nicas}~\cite[pp. 25-26]{Ni82},
 the $\Omega$-spectrum property implies the equivalence
 \[
 \mathbb{L}_q <j>  \  \longrightarrow  \  (\Omega \mathbb{L}_{q-1} )<j-1>  \   \cong  \  \mathbb{L}_{q-1} <j+1>.
 \]
 By iteration, one obtains the equivalence
 \[
 \mathbb{L}_{0} <n - | \sigma |>  \  \longrightarrow  \   \mathbb{L}_{n-m} <m- | \sigma | >
 \]
 as $\Delta$-sets.
 Note that
 \[
 \mathbb{L}_{0} <j> \cong \mathbb{Z} \times (G/TOP <j>),
 \]
 where
 $G/TOP<j>$
 denotes the singular complex of $j$-simplices of $G/TOP.$
 
 The face maps 
 \[
 \partial_{0}, ...,  \partial_{j}: \mathbb{L}_{0} <j>  \  \longrightarrow   \  \mathbb{L}_{0} <j-1>
 \]
 leave the $\mathbb{Z}$-components invariant.
 Since we are using dual complexes,
 the face maps can be written as follows
 \[
 \mathbb{L}_{n-m} <m-j>  \  \longrightarrow   \  \mathbb{L}_{n-m} <m-j-1>.
 \]
 
 \begin{lemma}\label{3.1}
 Suppose that  
 \[
 u : (V^m, W^m) \longrightarrow (\mathbb{L}_{n-m}, *)
 \]
 represents the element 
 $x \in H_n (B; \mathbb{L})$ as above,
 where
 $(S^m,S^m\setminus B)$
 is homotopy equivalent to the 
 simplicial model\index{simplicial model}
  for
 $(V^m,W^m).$  
 Then under the equivalence
 \[
 \mathbb{L}_{n-m} <m-j>  \  
 \cong  \ 
  \mathbb{L}_{0} <n-j> \
  \cong  \
  \mathbb{Z} \times (G/TOP <n-j>),
 \]
 $u(\sigma), u(\tau) \in \mathbb{L}_{n-m} $  
 determine the same
 $\mathbb{Z}$-component.
 \end{lemma}
 
 \begin{proof}
 The subface  
 $\sigma \prec \tau \in B$
 is a certain composition of face maps
 denoted by $\partial$,
 with the dual $\delta,$
 \[ 
 D(\tau, V^m) \prec  D(\sigma, V^m).
 \]
 
 The assertion of Lemma~\ref{3.1} now follows from the commutativity of the following diagram  
 \begin{equation}\label{d2}
\begin{tikzpicture}[baseline=-3.0cm,node distance=5.5cm, auto]
  \node (Xl) {$u(\tau) \in \mathbb{L}_{n-m}<m-| \tau |>$};
  \node (Fl) [right of=Xl] {$\mathbb{L}_{0}<n-| \tau |>$};
 \node (FFl) [right of=Fl] {};
  \node (Xl1) [below of=Xl] {$u(\sigma) \in \mathbb{L}_{n-m}<m-| \sigma |>$};
  \node (Fl1) [below of=Fl] {$\mathbb{L}_{0}<n-| \sigma |>$};
  \node (X) [node distance=3cm, below of=FFl] {$\mathbb{Z}$};  
  \draw[->, font=\small] (Xl) to node {} (Fl);
  \draw[->, font=\small] (Fl) to node {$pr$} (X);
  \draw[->, font=\small] (Xl1) to node {$\delta$} (Xl);  
  \draw[->, font=\small] (Fl1) to node [swap] {$pr$} (X);
  \draw[->, font=\small] (Xl1) to node {} (Fl1);
\end{tikzpicture}
\end{equation} 
 \end{proof} 
 \begin{corollary}\label{3.2}
 Consider any pair of simplices 
 $\sigma, \tau \in B$
 of 
 the 
 simplicial complex  $B$
 such that 
 $\sigma \cap \tau \ne \emptyset.$
 Then
 $u(\sigma), u(\tau) \in \mathbb{L}_{n-m} $  
determine the same 
$\mathbb{Z}$-component.\qed 
 \end{corollary}
 Moreover, Lemma~\ref{3.1} also implies the following corollary.
 \begin{corollary}\label{3.3}
 Suppose that  
 \[
 u, u': (V^m, W^m) \longrightarrow (\mathbb{L}_{n-m}, *)
 \]
 represent the same element 
 $x \in H_n (B; \mathbb{L}),$
 where
 $(S^m,S^m\setminus B)$
 is homotopy equivalent to the simplicial model for
 $(V^m,W^m).$
 Then for every
 simplex
 $\sigma \in B,$ 
 the elements
 $u(\sigma), u'(\sigma)    \in \mathbb{L}_{n-m} $  
determine the same 
$\mathbb{Z}$-component. 
 \end{corollary}
 
 \begin{proof}
 By hypothesis,
 $u \sim u',$
 so there is a cobordism map
 \[
 v:\Delta^1 \times (V^m, W^m)
 \longrightarrow
 (\mathbb{L}_{n-m}, *),
 \]
 i.e.
 \[
v (\Delta^1 \times \sigma) \in   \mathbb{L}_{n-m}<m-| \sigma | -1>, 
 \]
 where
 \[
 v (\partial_{0} \Delta^1 \times \sigma) = u(\sigma) 
 \quad
 \hbox{and}
 \quad
 v (\partial_{1} \Delta^1 \times \sigma) = u'(\sigma).
 \]
 Therefore under the face maps, we get the following diagram 
 \begin{equation}\label{d3}
\begin{tikzpicture}[baseline=-3.0cm, node distance=5.5cm, auto]
  \node (Xl) {$\mathbb{L}_{n-m}<m - | \sigma |>$};
  \node (Fl) [right of=Xl] {};
 \node (FFl) [right of=Fl] {};
  \node (Xl1) [below of=Xl] {$\mathbb{L}_{n-m}<m - | \sigma |>$};
  \node (Fl1) [below of=Fl] {};
  \node (X) [node distance=3cm, below of=FFl] {$\mathbb{L}_{0}<n- |\Delta^{1} \times \sigma |>$};
  \node (XX) [node distance=3cm, below of=Fl] {$\mathbb{L}_{n-m}<m- |\Delta^{1} \times \sigma |>$};
  \node (FFl1) [below of=FFl] {$\mathbb{Z}$};
  \draw[->, font=\small] (Xl) to node {$\delta_{0}$} (XX);  
  \draw[->, font=\small] (Xl1) to node {$\delta_{1}$} (XX);  
  \draw[->, font=\small] (X) to node {$pr$} (FFl1);
  \draw[->, font=\small] (XX) to node {} (X);
\end{tikzpicture}
\end{equation}  
 which proves the assertion of Corollary~\ref{3.3}.
 \end{proof}
 \section{Proof of Theorem~\ref{th2}}\label{s4}
 It follows from Lemma~\ref{3.1} and Corollaries~\ref{3.2} and \ref{3.3}
 that the 
 $\mathbb{Z}$-components depend only on 
 $x \in H_n (B; \mathbb{L}),$
 and that 
 $u(\sigma), u(\sigma')    \in \mathbb{L}_{n-m} $   
 define the same element if $\sigma$ and $\sigma'$
 can be connected by a chain
 $\sigma_1, ..., \sigma_r$
 of simplices
 such that
 $$\sigma_j \cap \sigma_{j+1}\neq \emptyset,
 \
 \hbox{for every}
 \
 j\in \{1,...,r-1 \}.$$
 
 This leads to the following theorem. 
 \begin{theorem}\label{3.4}
 Suppose that B is connected. Then the construction described above defines a map
 \[
 I: H_n (B; \mathbb{L}) \longrightarrow L_0 (\mathbb{Z}) = 1 + 8 \mathbb{Z}. 
 \]
 \end{theorem}
 
 \begin{remark}
 The $\mathbb{Z}$-component coming from the identification
 \[
 \mathbb{L}_{0} \cong \mathbb{Z} \times G/TOP
 \]
 is the $0$-dimensional signature invariant defind by {\sc Quinn}~\cite[Section 3.2]{Qu83}.
 \end{remark} 

The second component  
of the identification
\[
\mathbb{L}_{n-m}<m - | \sigma | >
\
\cong
\
\mathbb{L}_{0} < n - | \sigma | >
\
\cong
\
\mathbb{L}_{0} (\mathbb{Z}) \times (G/TOP < n - | \sigma | >)
\]
associates to 
$x \in H_n (B; \mathbb{L}),$
represented by the map,
where
$(V^m,W^m)$ 
is a pair of simplicial complexes homotopy equivalent to the pair
$(S^m,S^m \setminus B),$
\[
u: (V^m,W^m) \longrightarrow (\mathbb{L}_{n-m}, *),
\]
a family of adic normal degree one maps given by $u(\sigma)$ 
\begin{equation}\label{d4}
\begin{tikzpicture}[baseline=-1.0cm, node distance=4.5cm, auto]
\node (Xl) {};
\node (Fl) [right of=Xl] {};
\node (FFl) [right of=Fl] {};
\node (FFFl) [right of=FFl] {$V^m$};
\node (Xl1) [below of=Xl] {};
\node (Fl1) [below of=Fl] {};
\node (FFFl1) [below of=FFFl] {$B$};
\node (XXX) [node distance=3cm, below of=Xl] {$M^{n}_{\sigma}$};
\node (X) [node distance=3cm, below of=FFl] {$D(\sigma, V^m)$};
\node (XX) [node distance=3cm, below of=Fl] {$N^{n- | \sigma |}_{\sigma}$};
\node (FFl1) [below of=FFl] {};
\draw[->, font=\small] (X) to node {} (FFFl1);
\draw[->, font=\small] (X) to node {$incl.$} (FFFl); 
\draw[->, font=\small] (XXX) to node {} (XX);  
\draw[->, font=\small] (FFFl1) to node [swap] {$incl.$} (FFFl);
\draw[->, font=\small] (XX) to node {} (X);
\end{tikzpicture}
\end{equation}  
therefore it defines a map
\[
\Gamma: H_n(B;\mathbb{L}) \longrightarrow \mathcal{N}^{H}_{n}(B,\partial).
\]
Similarly, one obtains a map
\[
\Gamma^{+}: H_n(B;\mathbb{L}^{+}) \longrightarrow \mathcal{N}_{n}(B,\partial).
\]
Clearly,
the following diagram commutes
\begin{equation}\label{d5}
\begin{tikzpicture}[baseline=-2.0cm, node distance=4.5cm, auto]
  \node (Xl) {$H_n(B;\mathbb{L}^{+})$};
  \node (Fl) [right of=Xl] {$\mathcal{N}_{n}(B,\partial)$};
  \node (Xl1) [below of=Xl] {$H_n(B;\mathbb{L})$};
  \node (Fl1) [below of=Fl] {$\mathcal{N}^{H}_{n}(B,\partial)$};
  \draw[->, font=\small] (Xl) to node {$\Gamma^{+}$} (Fl);
  \draw[->, font=\small] (Xl) to node {} (Xl1);
  \draw[->, font=\small] (Fl) to node {} (Fl1);
  \draw[->, font=\small] (Xl1) to node {$\Gamma$} (Fl1);
\end{tikzpicture}
\end{equation}

\begin{summary}
For every $n \ge 6,$
one can construct the maps 
 \[
 I: H_n (B; \mathbb{L}) \longrightarrow L_0 (\mathbb{Z}) = 1 + 8 \mathbb{Z} 
 \]
and
\[
\Gamma: H_n(B;\mathbb{L}) \longrightarrow \mathcal{N}^{H}_{n}(B,\partial)
\]
via the simplicial equivalence
\[
\mathbb{L}_{n-m}<m - j >
\
\cong
\
\mathbb{L}_{0} < n - j >
\
\cong
\
\mathbb{L}_{0} (\mathbb{Z}) \times (G/TOP < n - j >).
\]

\end{summary}
This completes the proof of Theorem~\ref{th2}. \qed

\section*{Acknowledgements}
This research was supported by the Slovenian Research and Innovation Agency grants P1-0292, J1-4031, J1-4001,  N1-0278, N1-0114, and N1-0083. We thank John Bryant, Wolfgang L\"{u}ck,  Washington Mio, Shmuel Weinberger, and Michael Weiss for updating us regarding the doubts about the validity of {\sc Ferry-Pedersen} \cite[Theorem 16.6]{FP},
 which  in the last three decades was applied by many authors
(for details see e.g. {\sc Bryant-Ferry-Mio-Weinberger}~\cite{BFMW}),  and
which we also wanted to apply in our original plan.
 We thank the 
 editors 
 and 
 referees for comments and suggestions.

\end{document}